\documentclass[a4paper,12pt]{article}
\catcode`\"=13
\def"#1{\v#1}
\usepackage{graphicx,amsmath,amssymb,euscript,amsfonts}
\usepackage{textcomp}

\usepackage{pstricks, pst-node}
\usepackage{graphics,graphicx,graphpap}
\usepackage{pstricks-add}
\usepackage{makeidx}
\usepackage{epsfig}
\usepackage{color}
\usepackage{pstricks, pst-node}
\usepackage{graphics,graphicx,graphpap}
\usepackage{fancyhdr}
\usepackage{color}
\usepackage{multirow}

\newcommand{\qed}{\hfill\hbox{\rule{3pt}{6pt}}}

\newcommand{\T}{{\mathcal T}}

\newtheorem{theorem}{Theorem}[section]
\newtheorem{lemma}[theorem]{Lemma}

\date{}

\title{ Non-existence of antipodal cages\\ of even girth}
\author{ Slobodan Filipovski\thanks{Supported in part by the Slovenian Research Agency (research program P1-0285 and Young Researchers Grant).}\\
\small University of Primorska\\[-0.8ex]
\small Koper, Slovenia\\
\small\tt slobodan.filipovski@famnit.upr.si\\
}

\begin{document}
\maketitle

\begin{abstract}
The Moore bound $M(k,g)$ is a lower bound on the order of $k$-regular graphs of girth $g$ (denoted $(k,g)$-graphs). The excess $e$ of a $(k,g)$-graph of order $n$ is the difference $n-M(k,g).$ A $(k,g)$-cage is a $(k,g)$-graph with the fewest possible number of vertices, among all $(k,g)$-graphs.
A graph of diameter $d$ is said to be antipodal if, for any vertices $u, v, w$ such that $d(u,v)=d$ and $d(u, w)=d$, it follows that $d(v, w)=d$ or $v=w.$
In \cite{bigito} Biggs and Ito proved that any $(k,g)$-cage of even girth $g=2d\geq6$ and excess $e\leq k-2$ is a bipartite graph of diameter $d+1.$
In this paper we treat $(k,g)$-cages of even girth and excess $e\leq k-2.$ Based on a spectral analysis we give a relation between the eigenvalues of the adjacency matrix $A$ and the distance matrix $A_{d+1}$ of such cages. Moreover, following the methodology used in \cite{bigito} and \cite{pineda}, we prove the non-existence of the antipodal $(k,g)$-cages of excess $e$, where $k\geq e+2\geq4$ and $g=2d\geq14.$

\end{abstract}

\section{Introduction}
\quad A $(k,g)$-{\em graph} is a $k$-regular graph having girth $g$. A $(k,g)$-\emph{cage} is a $(k,g)$-graph of smallest order. 
The \textit{Cage Problem} or \textit{Degree/Girth Problem} calls for finding cages;  Tutte was the first to study $(k,g)$-cages,
 \cite{K}. A $(k,g)$-graph exists for any pair $(k, g),$ where $k \geq2$ and $g\geq 3$, see \cite{erdossachs} and \cite{sachs}. 
It is well known that the $(k,g)$-graphs have at least $M(k,g)$ vertices, where
\begin{equation} \label{moore1}
 M(k,g) = \left\{
\begin{array}{lc} 1 + k + k(k-1) +\cdot\cdot\cdot+ k(k-1)^{(g-3)/2},
                                                         & g \mbox{ odd, } \\
                              2 \left( 1 + (k-1)+ \cdot\cdot\cdot+ (k-1)^{(g-2)/2} \right),
                                                         & g \mbox{ even. }
\end{array}
\right.
\end{equation}
If $G$ is a $(k,g)$-graph of order $n$, then we define the \emph{excess} $e$ of $G$ to be $n-M(k,g).$
The graphs whose orders are equal to $M(k,g)$ (excess $0$) are called {\em Moore graphs}. Their classification has been completed except for the case $k=57$ and $g=5.$ The Moore graphs exist if $k=2$ and $ g \geq 3 $, $g=3$ and $ k \geq 2 $,
$g=4$ and $ k \geq 2 $, $g = 5$ and $k=2, 3, 7 $, or $g = 6, 8, 12$ and a generalized
$n$-gon of order $k-1$ exists, see  \cite{banito},  \cite{dam} and \cite{exojaj}.\\
The following three results concern the graphs of even girth.

\begin{theorem}[Biggs and Ito \cite{bigito}]\label{bigito}
Let $G$ be a $(k,g)$-cage of girth
$g = 2d \geq 6$ and excess $e$.
If $e \leq k-2$, then $e$ is even and $G$ is bipartite of diameter $d + 1$.
\end{theorem}
It is known that these graphs are partially distance-regular. More about almost-distance-regular graphs, see \cite{distgraphs}.
For the next theorem, let $D(k,2)$ denote the  incidence graph of a symmetric $(v,k,2)$-design.
\begin{theorem}[Biggs and Ito \cite{bigito}]\label{excess2}
Let $G$ be a $(k,g)$-cage of girth $g = 2d \geq 6$ and excess $2$. Then $g=6$, $G$ is a double-cover of $D(k,2)$, and $k \not \equiv 5, 7 \!\! \pmod{8}$.
\end{theorem}
\begin{theorem}[Jajcayov\' a, Filipovski and Jajcay \cite{ourpaper}]\label{exclude}
Let $ k \geq 6 $ and $ g = 2d > 6 $. No $(k,g)$-graphs of excess $4$ exist for parameters
$k,g$ satisfying at least one of the following conditions:
\begin{itemize}
\item[{\rm 1)}] $g=2p$, with $p\geq5$ a prime number, and $k \not \equiv 0, 1, 2 \!\! \pmod{p}$;
\item[{\rm 2)}] $ g=4\cdot3^{s}$ such that $ s\geq 4 $, and $k$ is divisible by $9$ but not by $ 3^{s-1}$;
\item[{\rm 3)}]  $ g=2p^{2}$, with $p\geq 5$ a prime number, and  $k \not \equiv 0, 1, 2 \!\! \pmod{p}$ and $k$ even;
\item[{\rm 4)}] $ g=4p$, with $p\geq 5 $ a prime number, and  $k \not \equiv 0, 1, 2,3,p-2 \!\! \pmod{p}$;
\item[{\rm 5)}] $g\equiv0\!\! \pmod {16}$, and $k\equiv3\!\! \pmod {g}.$
\end{itemize}
\end{theorem}

Let $k\geq4, g=2d\geq 6$ and let $G$ be a $(k,g)$-cage of excess $e\leq k-2$ and order $n$. Due to Theorem 1.1, we conclude that $G$ is a bipartite graph of diameter $d+1.$ Let $A$ be its adjacency matrix. For the integers $i$ with $0\leq i\leq d+1$, the \emph{$i$-distance matrix $A_{i}$} of $G$ is an $n\times n$ matrix such that the entry in position $(u,v)$ is $1$ if the distance between the vertices $u$ and $v$ is $i$, and zero otherwise.
 Using the spectral considerations as in \cite{cages}, in Section 2 we prove that the eigenvalues of $A(A_{1})$, other than $\pm k$, are the roots of the polynomial $H_{d-1}(x)+\lambda$; Theorem 2.3. Here, $H_{d-1}$ is the Dickson polynomial of the second kind with parameter $k-1$ and degree $d-1$, and $\lambda$ is an eigenvalue of the distance matrix $A_{d+1}$.

A graph of diameter $d$ is said to be \emph{antipodal} if, for any vertices $u, v, w$ such that $d(u,v)=d$ and $d(u, w)=d$, it follows that $d(v, w)=d$ or $v=w,$ (see \cite{L}).\\
Among the trivially antipodal graphs let us mention the $n$-dimensional cubes $Q_{n}$. These graphs are bipartite and have the antipodal property, since every vertex of $Q_{n}$ has a unique vertex at maximum distance from it.
Also, for $n\geq2$, the complete bipartite graph $K_{n,n}$ is antipodal. Here the antipodal partition is the same as the bipartition.
The dodecahedron is an example of trivially antipodal, but not bipartite graph. Examples of graphs which are non-trivially antipodal and not bipartite are the complete tripartite graphs $K_{n, n, n}$, which have diameter $2$, and the line graph of Petersen's graph, which has diameter $3$.
Motivated by Theorem 1.4, in this paper we address the question of the existence of the antipodal $(k,g)$-cages of even girth and excess $e\leq k-2.$ Employing the methodology used in \cite{Ban&Ito}, \cite{bigito} and \cite{pineda}, we prove the non-existence of the antipodal $(k,g)$-cages of excess $e$, with $k\geq e+2\geq4$ and $g=2d\geq14;$ Theorem 4.2.


 \begin{theorem}[Bannai and Ito \cite{Ban&Ito}] For $d\geq3,$ there exist no antipodal regular graphs with diameter $d+1$ and girth $2d+1.$
 \end{theorem}


\section{On $(k,g)$-cages of even girth and excess $e\leq k-2$}

\quad Let $k, g, d$ and $e$ be positive integers such that $k\geq e+2$ and $g=2d\geq6$. Let $G$ be a $(k,g)$-cage of excess $e$; Theorem 1.1 asserts that $e$ is even and $G$ is a bipartite graph of diameter $d+1.$ Let $f=\{u,
v\}$ be an arbitrary edge of $G$. Let $\T_u$ be the subgraph of $G$ induced by the set of vertices $x \in V(G)$ such
that $d(u,x)\leq \frac{g-2}{2}$ and $d(v,x) = d(u,x)+1$. It is easy to see that  $\T_u$ is a tree of depth $\frac{g-2}{2}$ rooted at $u$. In the same
way we can define the tree $\T_v$ to be the subgraph of $G$ induced by the set of vertices $x \in V(G)$ such that $d(v,x)\leq \frac{g-2}{2}$ and $d(u,x) =
d(v,x)+1$. Since $G$ is of girth $g=2d$, the trees ${\mathcal T}_u$ and ${\mathcal T}_v$ are disjoint. Let $\T_{uv}$ be the union of the trees
${\mathcal T}_u$ and ${\mathcal T}_v$ and the edge $f$. We call the graph $\T_{uv}$ a \emph{Moore tree of $G$ rooted at $f$}.
The graph $G$ must contain $e$ additional vertices $ w_1,w_2,...,w_{e-1},w_e$, which do not belong to $\T_{uv}$, that is, $d(w_{i},u)>\frac{g-2}{2}$ and $d(w_{i},v)>\frac{g-2}{2}$, for each $1\leq i\leq e.$
We call these vertices {\em the excess vertices with respect to $f$} and denote this set
$ X_f = \{ w_1,w_2,...,w_{e-1},w_e \} $; we call the edges not contained in the Moore tree
of $G$ {\em horizontal edges}.
Since $G$ is a bipartite graph, it contains no odd cycle; consequently there exists no edge between the excess vertices of the same partite set. Moreover, in order to balance the Moore tree $\T_{uv}$ and paring out the horizontal edges of $G$, we easily observe that half ($\frac{e}{2}$) of the excess vertices belong to the first, and the other half to the second partite set of $G.$ It implies that for each vertex of $V(G)$ there exist exactly $\frac{e}{2}$ vertices at distance $d+1$ from it.\\
In order to study the spectral properties of $G$, we define the following polynomials:


\begin{figure}[!h]
\begin{center}
  \includegraphics[width=\textwidth]{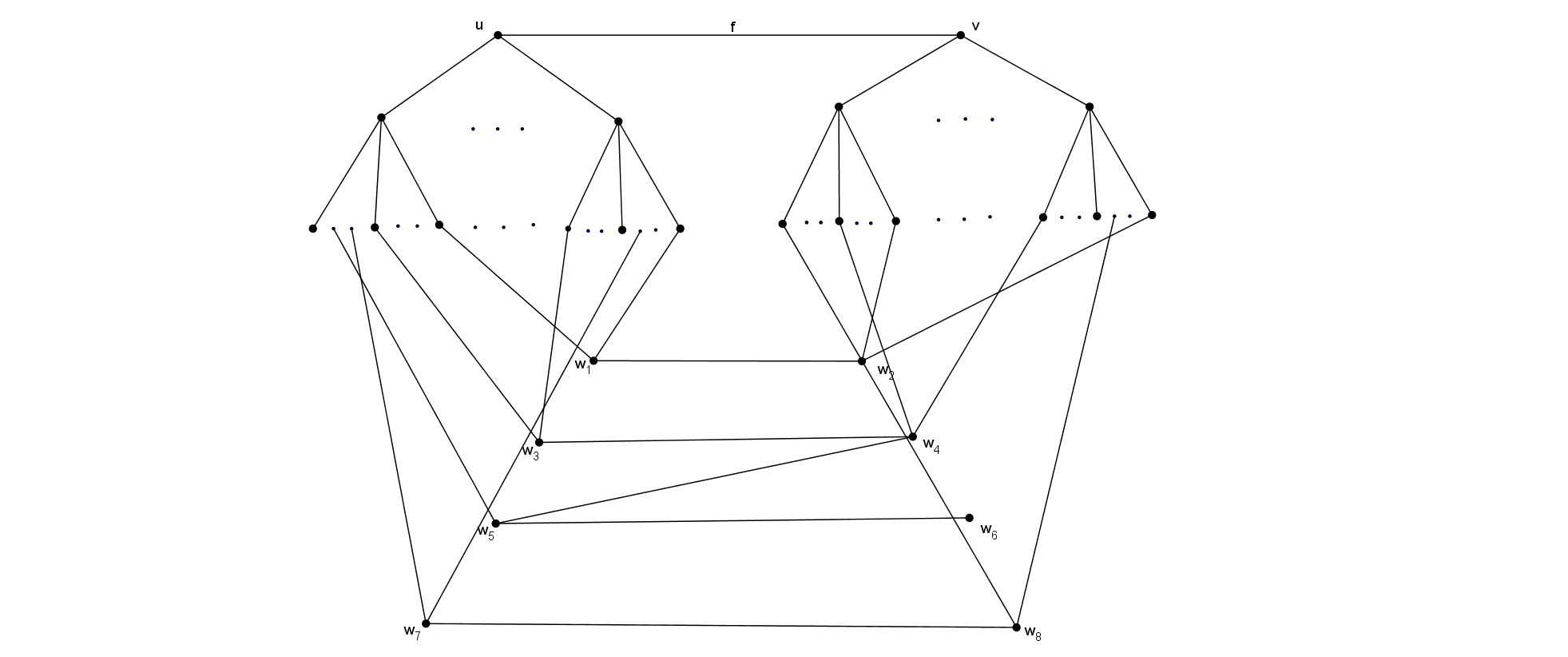}
  \end{center}
  \caption{The Moore tree and some of the horizontal edges in a potential (k, 6)-graph of
excess 8
}\label{picture}
\end{figure}

\begin{center}
$G_{0}(x)=1, \; G_{1}(x)=x+1;$
\end{center}
\begin{center}
$F_{0}(x)=1,\; F_{1}(x)=x,\; F_{2}(x)=x^{2}-k;$
\end{center}
\begin{center}
$H_{-2}(x)=-\frac{1}{k-1},\; H_{-1}(x)=0,\; H_{0}(x)=1,\; H_{1}(x)=x;$
\end{center}

\begin{equation}\label{2}
P_{i+1}(x)=xP_{i}(x)-(k-1)P_{i-1}(x) \mbox{ for }
\left\{
\begin{array}{lc} i\geq1, & \mbox{ if } P_{i}=G_{i}, \\
i\geq2, & \mbox{ if } P_{i}=F_{i}, \\
i\geq1, & \mbox{ if } P_{i}=H_{i}.
\end{array}
\right.
\end{equation}


The above defined polynomials have a close connection to the properties of a graph $G$. Namely, for $l < g$, the element $(F_{l}(A))_{x,y}$ counts the number of paths of length $l$ joining vertices $x$ and $y$ of $G$. Moreover, $G_{l}(A)$ counts the number of paths of length at most $l$ joining pairs of vertices in $G$. For more information about these polynomials see \cite{singleton}.

The next lemma is a generalization on Lemma 5 from \cite{cages}, where it is considered cages of even girth and excess $4$.
\begin{lemma}\label{vazna} Let $k\geq e+2 $ and $g=2d\geq6,$ and let $G$ be a $(k,g)$-cage of excess $e$. If $A$ is the adjacency matrix of $G$, then
$$F_{d}(A)=kA_{d}-A A_{d+1}.$$
\end{lemma}

\begin{proof} Let $f=\{u,v\}$ be a base edge of the Moore tree and let $f_{1}=\{w_{1}, w_{2}\}, f_{2}=\{w_{3}, w_{4}\},...,f_{\frac{e}{2}}=\{w_{e-1}, w_{e}\}$ be the edges of the subgraph induced by $X_f$. Also, let us assume that $d(u,w_{1})=d(u,w_{3})=...=d(u,w_{e-1})=d$ and
$d(u,w_{2})=d(u,w_{4})=...=d(u,w_{e})=d+1.$ Let $l_{i}$ be the number of edges between $w_{i}$ and the leaves of ${\mathcal T}_u$ and ${\mathcal T}_v$, where $1\leq i\leq e.$ We consider the case when the excess vertices do not share common neighbour among the leaves of ${\mathcal T}_u$ and ${\mathcal T}_v$. The opposite case one can prove in a similar way. By the definition of $F_{i}(x)$, we have $(F_{d}(A))_{u,w_{i}}=l_{i}$, for each odd $i$, $1\leq i\leq e-1.$ Considering the vertices at distance $d$ from $u$, there are also the $(k-1)^{d-1}$ leaves of ${\mathcal T}_v$. For $l_{2}+l_{4}+...+l_{e}$ of these vertices, there exist $k-1$ paths of length $d$ from $u$ to them. Namely, they are the vertices adjacent to $w_{2},w_{4},...,w_{e-2}$ or $w_{e}.$ For all the other leaves, there are $k$ paths between them and $u$. Thus, $(F_{d}(A))_{u,s}=0$ if $d(u,s)\neq d$, $(F_{d}(A))_{u,s}=k$ if $s$ is a leaf of ${\mathcal T}_v$ and not adjacent to $w_{2}, w_{4},...,w_{e}$, $(F_{d}(A))_{u,s}=k-1$ if $s$ is a leaf of ${\mathcal T}_v$ and adjacent to one of $w_{2}, w_{4},...,w_{e}$, and $(F_{d}(A))_{u,w_{i}}=l_{i}$, for each odd $i$; $1\leq i\leq e-1.$  For the matrix $kA_d$ we have $(kA_{d})_{u,s}=k$ if $d(u,s)=d$ and $(kA_{d})_{u,s}=0$ if $d(u,s)\neq d.$
Now, let $s$ be a vertex of $G$ such that $d(u,s)=d$ and $s$ is adjacent to one of $w_{2}, w_{4},...,w_{e}.$ If $s$ is a vertex among the vertices $w_{1}, w_{3},...,w_{e-1}$, then it is easy to see that $(A A_{d+1})_{u,s}=k-l_{i}.$ On the other hand, since $s$ is adjacent to ${\mathcal T}_u$ through $k-2$ different horizontal edges, it follows that, between the $k-1$ branches of ${\mathcal T}_u$, there exists one sub-branch that is not adjacent to $s$ through a horizontal edge. Let $s_{1}$ be the root of that sub-branch. Then, $d(s,s_{1})=d+1$ and $d(u,s_{1})=1$, which implies $(A)_{u,s_{1}}=1$ and $(A_{d+1})_{s_{1}, s}=1.$
Let $s_{i}$, $2\leq i\leq \frac{e}{2}$ be the remaining vertices at distance $d+1$ from $s$. Because all neighbours of $u$, except $s_{1}$, are at distance smaller than $d+1$ from $s$, we have $(A)_{u,s_{i}}=0$ and $(A_{d+1})_{s_{i}, s}=1,$ for $2\leq i\leq \frac{e}{2}.$ Thus $(AA_{d+1})_{u,s}=1.$
 If $s$ is a vertex of $G$ such that $d(u,s)=d$ and $s$ is not adjacent to $w_{2}, w_{4},...,w_{e}$, then the distance between $s$ and the neighbours of $u$ is $d-1$. In this case, $(AA_{d+1})_{u,s}=0.$
If $d(u,s)\neq d$, then the distance between $s$ and the neighbours of $u$ is different from $d+1$, and therefore $(AA_{d+1})_{u,s}=0. $
The required identity follows from summing up the above conclusions.
\qed
\end{proof}
\medskip

 Based on the previous lemma and the properties of the polynomials $G_{i}$, $H_{i}$ and $F_{i}$, we obtain the next two results. Theorem 2.3 is the main result in this section; it gives a relationship between the eigenvalues of the matrices $A$ and $A_{d+1}$.
 We omitted their proofs because they follow analoglously like in Lemma 6 and Theorem 7 from \cite{cages}.

\begin{lemma}\label{identity} Let $k\geq e+2\geq4$ and $g=2d\geq6$, and let $G$ be a $(k,g)$-cage of excess $e$. If $A$ is the adjacency matrix of $G$ and $J$ is the all-ones matrix, then
$$kJ=(A+kI)(H_{d-1}(A)+A_{d+1}).$$
\end{lemma}

\begin{theorem} \label{eigenvalues}
If $\theta(\neq\pm k)$ is an eigenvalue of $A$, then
$$H_{d-1}(\theta)=-\lambda,$$
where $\lambda$ is an eigenvalue of $A_{d+1}$.
\end{theorem}

\section{Spectral analysis of the antipodal cages of even girth and small excess}

\quad In this section we study the spectral properties of the antipodal $(k,g)$-cages of even girth $g=2d\geq 6$ and excess at most $k-2$. Let $G$ be such graph, $A$ be its adjacency matrix and let $n$ be the order of $G$. Recall, $G$ is a bipartite graph of diameter $d+1$. Let $V_{1}$ and $V_{2}$ be the partitions of $G$. If $d$ is an even number, then any two vertices of $V(G)$ at distance $d+1$ belong to a different partite set. Clearly, this case is not possible considering antipodal bipartite graphs. Therefore, for the rest of the paper we assume $d$ odd.
Since for each vertex $u \in V(G)$ there exist exactly $\frac{e}{2}$ vertices at diameter distance $d+1$, (they are the excess vertices of the same partite set), we observe that the antipodal graph of $G$ is a disjoint union of $K_{\frac{e}{2}+1}$-complete graphs, and consequently, the distance matrix $A_{d+1}$ is an adjacency matrix of a disjoint union of $K_{\frac{e}{2}+1}$-complete graphs. Let $c$ be the number of such complete graphs. Obviously $c=\frac{2n}{e+2}.$  The spectrum of the disjoint union of $c$ complete graphs of order $\frac{e}{2}+1$ is known and determined by $\{(\frac{e}{2})^{c}, (-1)^{n-c}\}$ (see Propos. 6 in \cite{vandam}). Applying this result in Theorem \ref{eigenvalues}, we are in a position to determine the spectrum of $A$.

\begin{theorem} \label{roots}
If $\theta(\neq\pm k)$ is an eigenvalue of $A$, then
\begin{equation}\label{eq1}
H_{d-1}(\theta)-\epsilon=0,
\end{equation}
where $\epsilon=-\frac{e}{2}, 1$.
\end{theorem}
The roots of $H_{d-1}(x)$ are equal to $2\sqrt{k-1}\cos \frac{i\pi}{d}$ for $i=1,...,d-1$, (see \cite{singleton}). Therefore we assume $x=-2\sqrt{k-1}\cos \phi, 0<\phi<\pi$. Let $s=\sqrt{k-1}.$ Then we have
$$H_{d-1}(x)=(-s)^{d-1}\frac{\sin d\phi}{\sin \phi}.$$
Putting $\phi=\frac{i\pi-\alpha}{d}$, as suggested in \cite{Ban&Ito} and \cite{bigito}, we transform the equation (\ref{eq1}) as follows
\begin{equation}\label{relation}
\sin\alpha-\eta_{i} s^{-d+1}\sin\left(\frac{i\pi-\alpha}{d}\right)=0,
\end{equation}
where $\eta_{i}=\epsilon(-1)^{d+i}.$
The following result follows similarly as Lemma 3.3 from \cite{bigito} and Lemma 2.2 from \cite{pineda}.

\begin{lemma} The equation (\ref{eq1}) has $d-1$ distinct roots $\theta_{1}<\theta_{2}<...<\theta_{d-1}$, with $\theta_{i}=-2s\cos \phi_{i}, (0<\phi_{i}<\pi).$ If we set $\phi_{i}=\frac{i\pi-\alpha_{i}}{d}$ then
\begin{center}
$ 0<\alpha_{i}<min\{s^{-d+1}\phi_{i}, s^{-d+1}(\pi-\phi_{i})\}$ if $\eta_{i}=1;$
\end{center}

\begin{center}
$ max\{-s^{-d+1}\phi_{i}, -s^{-d+1}(\pi-\phi_{i})\}<\alpha_{i}<0$ if $ \eta_{i}=-1;$
\end{center}

\begin{center}
$ 0<\alpha_{i}<min\{\frac{e}{2}s^{-d+1}\phi_{i}, \frac{e}{2}s^{-d+1}(\pi-\phi_{i})\}$ if $ \eta_{i}=\frac{e}{2};$
\end{center}

\begin{center}
$ max\{-\frac{e}{2}s^{-d+1}\phi_{i}, -\frac{e}{2}s^{-d+1}(\pi-\phi_{i})\}<\alpha_{i}<0,$ if $ \eta_{i}=-\frac{e}{2}.$
\end{center}
From the bounds of $\alpha_{i}$ we derive the bounds of $\phi{_i}$ as follows.
\begin{center}
$\frac{i\pi}{d+s^{-d+1}}<\phi_{i}<\frac{i \pi}{d}$ if $\eta_{i}=1$;
\end{center}
\begin{center}
$\frac{i\pi}{d}<\phi_{i}<\frac{i \pi}{d-s^{-d+1}}$ if $ \eta_{i}=-1$;
\end{center}
\begin{center}
$\frac{i\pi}{d+\frac{e}{2} s^{-d+1}}<\phi_{i}<\frac{i \pi}{d}$ if $\eta_{i}=\frac{e}{2};$
\end{center}
\begin{center}
$\frac{i\pi}{d}<\phi_{i}<\frac{i \pi}{d-\frac{e}{2}s^{-d+1}}$ if $\eta_{i}=-\frac{e}{2}.$
\end{center}
\end{lemma}
%

We claim that $tr (A^{q})=n(B_{d}^{q})_{0,0}$ for $q=0, 1,...,2d-1,$ where
$$B_{D}=\left(
  \begin{array}{cccccccc}
    0 & 1 & & &  &  &  &  \\
    k & 0 & 1 & &  &  &0  & \\
     & k-1 & 0 & 1 &  &  &  &  \\
     &  & k-1 & 0 & 1 &  &  &  \\
    & &  &  & \ddots   & \ddots  &  &  \\
     &  &  &  & \ddots  &  &  &  \\
     & 0 &  &  &  & k-1  & 0  & k  \\
     &  &  &  &  &  & k-1 & 0  \\
  \end{array}
\right)$$
 is the $(D+1)\times (D+1)$ intersection matrix of a Moore bipartite graph of degree $k$, diameter $D$ and of girth $2D$, (see \cite{pineda}).
If $q<g(G)$, the number of closed walks of length $q$ that start from a fixed vertex $u$ is independent of the vertex $u$ and the excess.
Furthermore, the entry $(B_{\lceil\frac{g(G)}{2}\rceil}^{q})_{0,0}$ gives this number, where $(B_{i}^{q})_{0,0}$ is the $(0,0)$-entry of $B_{i}^{q},$
(see \cite{godsil}). The number of closed walks of length $q$ in $G$ is given by $tr(A^{q})$. Since $G$ is a bipartite graph, it follows that $G$ contains no closed walk of odd length. Thus, $tr(A^{q})=n(B_{d}^{q})_{(0,0)}$ for $q=1, 3,...,2d-3, 2d-1.$ Moreover, since the girth of $G$ is $2d$ we obtain $tr(A^{q})=n(B_{d}^{q})_{(0,0)}$ for $q=0, 1,...,2d-1.$
\begin{theorem} Let $\theta$ be a root of $H_{d-1}(x)-\epsilon.$ The multiplicity $m(\theta)$ of $\theta$ in $G$, $\theta\neq \pm k$, is given by
\begin{equation}\label{important}
m(\theta)=\frac{n e k(k-1) H_{d-2}(\theta)}{2 \epsilon (2\epsilon+\frac{e}{2}-1) H_{d-1}^{'}(\theta)(k^{2}-\theta^{2})}.
\end{equation}
\end{theorem}
\begin{proof} In order to compute the multiplicity of an eigenvalue $\theta$ of $G$, we employ the same approach from \cite{Ban&Ito}, \cite{bigito} and \cite{pineda}. Let $\xi(x)=(x^{2}-k^{2})(H_{d-1}(x)+\frac{e}{2})(H_{d-1}(x)-1)$ and $\xi_{\theta}(x)=\frac{\xi(x)}{x-\theta}.$ Since $\xi(A)=0,$ it follows $m(\theta)=\frac{tr(\xi_{\theta}(A))}{\xi_{\theta}(\theta)}.$\\
As $deg(H_{d-1}(x))=d-1$ we obtain that $deg (\xi_{\theta}(x))=2d-1$. Therefore, let us assume $\xi_{\theta}(x)=x^{2d-1}+a_{2d-2}x^{2d-2}+...+a_{1}x+a_{0}.$ Hence,
$$ tr(\xi_{\theta}(A))=tr(A^{2d-1})+a_{2d-2}tr(A^{2d-2})+...+ a_{1}tr(A)+a_{0}tr(I_{n}).$$
Since $tr(A^{q})=n(B_{d}^{q})_{0,0}$ for $0\leq q\leq 2d-1$, we have
$$tr(\xi_{\theta}(A))=n(\xi_{\theta}(B_{d}))_{0,0}.$$
The polynomial $(x^{2}-k^{2})H_{d-1}(x)$ is a minimal polynomial of $B_{d}$, (see \cite{singleton}). It implies
$$\xi_{\theta}(B_{d})=-\frac{e}{2}\frac{B_{d}^{2}-k^{2}I_{n}}{B_{d}-\theta I_{n}}.$$
Setting $L_{i+1}(x)=\frac{x^{2}-k^{2}}{x-\theta}(H_{i}(x)-H_{i}(\theta))$ for $i=0,...,d-1$, we get
$$L_{d} (B_{d})=-H_{d-1}(\theta)\frac{B_{d}^{2}-k^{2} I_{n}}{B_{d}-\theta I_{n}}=-\epsilon \frac{B_{d}^{2}-k^{2} I_{n}}{B_{d}-\theta I_{n}}.$$
Therefore, $\xi_{\theta}(B_{d})=\frac{e}{2\epsilon} L_{d}(B_{d}).$\\
Calculating the derivative of $(x-\theta)\xi_{\theta}(x)$, that is, $((x-\theta)\xi_{\theta}(x))^{'}=((x^{2}-k^{2})(H_{d-1}(x)+\frac{e}{2})(H_{d-1}(x)-1))^{'}$, we have $\xi_{\theta}(\theta)=(2\epsilon+\frac{e}{2}-1) H_{d-1}^{'}(\theta)(\theta^{2}-k^{2}).$ Thus
$$m(\theta)=\frac{n e}{2\epsilon(2\epsilon+\frac{e}{2}-1)}\frac{(L_{d}(B_{d}))_{0,0}}{H_{d-1}^{'}(\theta)(\theta^{2}-k^{2})}.$$
In \cite{pineda} was proven that $(L_{d}(B_{d}))_{0,0}=-k(k-1)H_{d-2}(\theta).$ Substituting it in the previous expression we obtain
$$m(\theta)=\frac{n e k(k-1)}{2\epsilon(2\epsilon+\frac{e}{2}-1)}\frac{H_{d-2}(\theta)}{H_{d-1}^{'}(\theta)(k^{2}-\theta^{2})}.$$
\end{proof}
\qed

\subsection{Multiplicities as function of $\cos \phi$}

Let $\theta$ be a root of $H_{d-1}(x)-\epsilon$ and let $\theta=-2s \cos \phi, 0<\phi<\pi$. We express the multiplicity of $\theta$, $m(\theta)$, as a function of $\cos \phi$. For that purpose we define the following functions $f(z), g_{1}(z), g_{2}(z)$ and $g_{3}(z).$
$$f(z)=\frac{4s^{2}(1-z^{2})}{k^{2}-4s^{2}z^{2}};$$

$$g_{1}(z)= \frac{k(k-1)(\sqrt{1-s^{-2d+2}(1-z^{2})}+s^{-d+1}z)}{d\sqrt{1-s^{-2d+2}(1-z^{2})}+s^{-d+1}z};$$

$$g_{2}(z)= \frac{k(k-1)(\sqrt{1-\frac{e^{2}}{4}s^{-2d+2}(1-z^{2})}-\frac{e}{2}s^{-d+1}z)}{d\sqrt{1-\frac{e^{2}}{4}s^{-2d+2}(1-z^{2})}-\frac{e}{2}s^{-d+1}z};$$

$$g_{3}(z)= \frac{k(k-1)(\sqrt{1-\frac{e^{2}}{4}s^{-2d+2}(1-z^{2})}+\frac{e}{2}s^{-d+1}z)}{d\sqrt{1-\frac{e^{2}}{4}s^{-2d+2}(1-z^{2})}+\frac{e}{2}s^{-d+1}z}.$$

\begin{lemma} For either value of $\epsilon$, if we set $\theta_{i}=-2s\cos \phi_{i}$ for $1\leq i\leq d-1$, then
$$m(\theta_{i})=\frac{n e}{4s^{2}(\frac{e}{2}+1)}f(\cos \phi_{i})g_{1}(\eta_{i}\cos\phi_{i}), \mbox{ if $\epsilon=1$};$$
$$m(\theta_{i})=\frac{n}{2s^{2}(\frac{e}{2}+1)}f(\cos \phi_{i})g_{2}(\cos\phi_{i}), \mbox{ if $\epsilon=-\frac{e}{2}$ and $i$ is odd};$$
$$m(\theta_{i})=\frac{n}{2s^{2}(\frac{e}{2}+1)}f(\cos \phi_{i})g_{3}(\cos\phi_{i}), \mbox{ if $\epsilon=-\frac{e}{2}$ and $i$ is even}.$$
\end{lemma}
\begin{proof} The derivative of $H_{d-1}(x)$ is computed in \cite{pineda}. We have
$$H_{d-1}^{'}(\theta_{i})=\frac{(-s)^{d-1}(-1)^{i}}{2s \sin^{2}\phi_{i}}(d\cos \alpha_{i}+\eta_{i}s^{-d+1}\cos\phi_{i}).$$
Substituting $H_{d-2}(\theta_{i})=(-s)^{d-2}(-1)^{i+1}\frac{\sin (\phi_{i}+\alpha_{i})}{\sin \phi_{i}}$ and $H_{d-1}^{'}(\theta_{i})$ in (\ref{important}), we obtain
$$m(\theta_{i})= \frac{ n ek (k-1) \sin \phi_{i} \sin(\phi_{i}+\alpha_{i})}{\epsilon (2\epsilon+\frac{e}{2}-1)(k^{2}-\theta_{i}^{2})(d \cos \alpha_{i}+\eta_{i}s^{-d+1}\cos \phi_{i})}.$$
The equation (\ref{relation}) yields $\sin(\phi_{i}+\alpha_{i})=\sin \phi_{i}(\cos \alpha_{i}+\eta_{i}s^{-d+1}\cos \phi_{i})$. Hence
$$ m(\theta_{i})= \frac{ ne \sin^{2} \phi_{i} }{\epsilon (2\epsilon+\frac{e}{2}-1)(k^{2}-\theta_{i}^{2})}\frac {k(k-1)(\cos \alpha_{i}+\eta_{i}s^{-d+1}\cos \phi_{i})}{(d \cos \alpha_{i}+\eta_{i}s^{-d+1}\cos \phi_{i})}.$$
By equation (\ref{relation}) and Lemma 3.2, as $k, d\geq 3$, it follows that if $\eta_{i}=1$ or $\eta_{i}=\frac{e}{2}$ then $0<\alpha_{i}<\frac{\pi}{2}.$ Similarly, if $\eta_{i}=-1$ or $\eta_{i}=- \frac{e}{2}$, then $-\frac{\pi}{2}<\alpha_{i}<0.$ Therefore $\cos \alpha_{i}>0$, and thus, $\cos \alpha_{i}=\sqrt{1-\eta_{i}^{2}s^{-2d+2}(1-\cos^{2}\phi_{i})}$. It implies
$$ m(\theta_{i})= \frac{ ne }{4s^{2}\epsilon (2\epsilon+\frac{e}{2}-1)}\frac{4s^{2}(1-\cos^{2}\phi_{i})}{k^{2}-4s^{2}\cos^{2}\phi_{i}}\frac {k(k-1)(\sqrt{1-\eta_{i}^{2}s^{-2d+2}(1-\cos^{2}\phi_{i})}+\eta_{i}s^{-d+1}\cos \phi_{i})}{(d \sqrt{1-\eta_{i}^{2}s^{-2d+2}(1-\cos^{2}\phi_{i})}+\eta_{i}s^{-d+1}\cos \phi_{i})}.$$
Using the formulas for $f, g_{1}, g_{2}$ and $g_{3}$ we get the desired result.
\qed
\end{proof}
\medskip

The following two lemmas concern the monotonicity of $f, g_{1}, g_{2}$ and $g_{3}.$ The first lemma is given in \cite{bigito} and \cite{pineda} (Lemma 3.5 and Lemma 4.1).

\begin{lemma} For $k\geq3$ and $\mid \!\! z \!\! \mid<1$ the function $f(z)$ is even and concave down.
\end{lemma}

\begin{lemma} For $k\geq3$, $d\geq3$ and $\mid \!\! z \!\! \mid<1$, the monotonicity of $g_{1}(z), g_{2}(z)$ and $g_{3}(z)$ behave as follows.
\begin{itemize}
\item[{\rm (1)}] $g_{1}(z)$ is monotonic increasing;
 \item[{\rm (2)}] $g_{2}(z)$ is monotonic decreasing;
 \item[{\rm (3)}] $g_{3}(z)$ is monotonic increasing;
\end{itemize}
\end{lemma}

\begin{proof}
 \begin{itemize}
 \item[{\rm (1)}] It is suffice to prove that $g_{1}^{'}(z)$ is positive on the interval $(-1,1)$. We have
$$g_{1}^{'}(z)= \frac{k(k-1)(d-1)s^{-d+1}(1-s^{-2d+2})}{\sqrt{1+s^{-2d+2}(-1+z^{2})}(d \sqrt{1+s^{-2d+2}(-1+z^{2})}+s^{-d+1}z)^{2}}>0.$$
 \item[{\rm (2)}] In this case we prove that $g_{2}^{'}(z)$ is negative on the interval $(-1,1)$.
 $$g_{2}^{'}(z)= \frac{-\frac{e}{2}s^{-d+1}k(k-1)(d-1)(1-\frac{e^{2}}{4}s^{-2d+2})}{\sqrt{1+\frac{e^{2}}{4}s^{-2d+2}(-1+z^{2})}(d \sqrt{1+\frac{e^{2}}{4}s^{-2d+2}(-1+z^{2})}-\frac{e}{2}s^{-d+1}z)^{2}}.$$
Since $k, d\geq3$ and $k\geq e+2$, we easily conclude that $ \frac{e^{2}}{4}s^{-2d+2}<1$ and $ \mid \!\!  \frac{e^{2}}{4}s^{-2d+2}(-1+z^{2})\!\! \mid <1$. 
\item[{\rm (3)}] It follows from the same reasoning as (2).
\end{itemize}
\qed
\end{proof}

\section{Main result}
Let $\lambda_{1}<\lambda_{2}<...<\lambda_{d-1}$ be the roots of $H_{d-1}(x)+\frac{e}{2}$, and let $\mu_{1}<\mu_{2}<...<\mu_{d-1}$ be the roots of $H_{d-1}(x)-1.$
\begin{lemma} Let $\lambda_{1},...,\lambda_{d-1}$ and $\mu_{1},...,\mu_{d-1}$ be defined as above. If $k\geq3$ and $d\geq3$ is an odd number, then
\begin{itemize}
\item[{\rm (1)}] $m(\lambda_{i})=m(\lambda_{d-i})$ and  $m(\mu_{i})=m(\mu_{d-i})$, for $1\leq i\leq d-1$;
 \item[{\rm (2)}]  $m(\mu_{2})<m(\mu_{i})$ for $3\leq i\leq d-3$, and $m(\lambda_{1})<m(\lambda_{i})$, for $2\leq i\leq d-2.$
\end{itemize}
\end{lemma}
\begin{proof}
\begin{itemize}
\item[{\rm (1)}] If $d$ is odd $H_{d-1}(-x)=H_{d-1}(x)$. Therefore $\theta$ is a root of $H_{d-1}(x)-\epsilon$, if and only if, $-\theta$ is a root of $H_{d-1}(x)-\epsilon,$ (see [1]). Then $\lambda_{i}+\lambda_{d-i}=\mu_{i}+\mu_{d-i}=0$. By checking (\ref{important}) and using $H_{d-2}(-x)=-H_{d-2}(x)$, we obtain $m(\lambda_{i})=m(\lambda_{d-i})$ and $m(\mu_{i})=m(\mu_{d-i})$ for each $1\leq i\leq d-1$.
\item[{\rm (2)}] Since $\mu_{i}$ is a root of $H_{d-1}(x)-1$, we have $\epsilon=1.$ According to Lemma 3.2 let us set $\mu_{i}=-2s \cos \phi_{i}$, for $1\leq i\leq d-1$. In this case $\eta_{i}=\epsilon (-1)^{d+i}=(-1)^{i+1}.$
    Since $-\mu_{2}=\mu_{d-2}$ we obtain $-\cos \phi_{2}=\cos \phi_{d-2}.$ Now, for $3\leq i\leq d-3$, we have $-\cos \phi_{2}=\cos \phi_{d-2}< \cos \phi_{i}< \cos \phi_{2}$. Since $f$  is even and concave down function we have
    $$f(\cos  \phi_{2})<f(\cos \phi_{i}) \mbox{ for $3\leq i\leq d-3$.}$$
    The inequality $\cos \phi_{i}<|\cos \phi_{2}|$ and the fact that $g_{1}(z)$ is a monotonic increasing function yield $g_{1}(\eta_{2}\cos \phi_{2})=g_{1}(-\cos \phi_{2})<g_{1}(\pm \cos \phi_{i})$.\\ Therefore, for $3\leq i\leq d-3$, we conclude
     $$m(\mu_{2})=\frac{n e}{4s^{2}(\frac{e}{2}+1)}f(\cos \phi_{2})g_{1}(\eta_{2} \cos \phi_{2})<\frac{n e}{4s^{2}(\frac{e}{2}+1)}f(\cos \phi_{i})g_{1}(\pm \cos \phi_{i})=m(\mu_{i}).$$
     We proceed similarly when $\lambda_{i}$ is a root of $H_{d-1}(x)+\frac{e}{2}$. In this case $\epsilon=-\frac{e}{2}$ and $\eta_{i}=\frac{e}{2}(-1)^{i}.$
     Again let $\lambda_{i}=-2s\cos \phi_{i},$ for $1\leq i\leq d-1$. Following the same reasoning as above we have
       $f(\cos  \phi_{1})<f(\cos \phi_{i})$  for $2\leq i\leq d-2$.\\
     Now, let $i$ be an odd number such that $3\leq i\leq d-2$. For such $i$ we note $\eta_{i}=-\frac{e}{2}<0$.
       Hence $g_{2}(z)$ is a monotonic decreasing, and therefore, $\cos \phi_{i}<\cos \phi_{1}$ yields $g_{2}(\cos \phi_{1})<g_{2}(\cos \phi_{i}).$
       Thus, for odd $i$ such that $3\leq i\leq d-2$ we have
       $$m(\lambda_{1})=\frac{n}{2s^{2}(\frac{e}{2}+1)}f(\cos \phi_{1})g_{2}(\cos \phi_{1})<\frac{n}{2s^{2}(\frac{e}{2}+1)}f(\cos \phi_{i})g_{2}( \cos \phi_{i})=m(\lambda_{i}).$$
       Since $m(\lambda_{i})=m(\lambda_{d-i})$ occurs $m(\lambda_{1})<m(\lambda_{i})$ for each $2\leq i\leq d-2.$
\end{itemize}
\qed
\end{proof}

Based on Lemma 4.1, we are ready to give the main result in this paper.

\begin{theorem} Let $k\geq e+2\geq4$ and $g=2d\geq14$ be fixed. There exists no antipodal $(k,g)$-cage of excess $e$.
\end{theorem}
\begin{proof} Since $m(\mu_{2})<m(\mu_{i})$ for $ 3\leq i\leq d-3$ and $\mu_{1}+\mu_{d-1}=0$, we obtain that $\mu_{2}$ and $\mu_{d-2}=-\mu_{2}$ are either conjugate quadratic irrationals or integers. Therefore, $\mu_{2}^{2}$ is an integer. Analogously, $\lambda_{2}^{2}$ is an integer. Hence $\lambda_{2}^{2}-\mu_{2}^{2}$ is an integer number.
By Lemma 3.2 we have
$$-2s \cos \frac{2 \pi}{d}<\mu_{2}<-2s \cos \frac{2 \pi}{d-s^{-d+1}}$$
$$-2s \cos \frac{2 \pi}{d+\frac{e}{2}s^{-d+1}}<\lambda_{2}<-2s \cos \frac{2 \pi}{d}.$$
Then, as $ \lambda_{2}^{2}>4s^{2}\cos^{2} \frac {2 \pi}{d}$
and $ \mu_{2}^{2}<4s^{2}\cos^{2} \frac {2 \pi}{d}$, we have that
$\lambda_{2}^{2}-\mu_{2}^{2}>0.$
Furthermore, as $ \lambda_{2}^{2}<4s^{2}\cos^{2} \frac {2 \pi}{d+\frac{e}{2}s^{-d+1}}$ and $ \mu_{2}^{2}>4s^{2}\cos^{2} \frac {2 \pi}{d-s^{-d+1}}$, we have that
$$\lambda_{2}^{2}-\mu_{2}^{2}<4s^{2}\left(\cos^{2}\frac{2\pi}{d+\frac{e}{2}s^{-d+1}}-
\cos^{2}\frac{2\pi}{d-s^{-d+1}}\right)=
4s^{2}\left(\sin^{2}\frac{2\pi}{d-s^{-d+1}}-\sin^{2}\frac{2\pi}{d+\frac{e}{2}s^{-d+1}}\right)=$$
$$=4s^{2}\left(\sin\frac{2\pi}{d-s^{-d+1}}-\sin\frac{2\pi}{d+\frac{e}{2}s^{-d+1}}\right)\left(\sin\frac{2\pi}{d-s^{-d+1}}+\sin\frac{2\pi}{d+\frac{e}{2}s^{-d+1}}\right)=$$
$$=16 s^{2} \sin \left(\frac{\pi}{d-s^{-d+1}}-\frac{\pi}{d+\frac{e}{2}s^{-d+1}}\right)\cos \left(\frac{\pi}{d-s^{-d+1}}+\frac{\pi}{d+\frac{e}{2}s^{-d+1}}\right)\cdot$$
 $$\cdot \sin \left(\frac{\pi}{d-s^{-d+1}}+\frac{\pi}{d+\frac{e}{2}s^{-d+1}}\right)
\cos \left(\frac{\pi}{d-s^{-d+1}}-\frac{\pi}{d+\frac{e}{2}s^{-d+1}}\right)<$$
$$<16s^{2}\pi^{2}\left(\frac{1}{(d-s^{-d+1})^{2}}-\frac{1}{(d+\frac{e}{2}s^{-d+1})^{2}}\right)=
\frac{16\pi^{2}(\frac{e}{2}+1)s^{-d+3}(2d+(\frac{e}{2}-1)s^{-d+1})}{(d-s^{-d+1})^{2}(d+\frac{e}{2}s^{-d+1})^{2}}.$$
Since $(d-s^{-d+1})^{2}>(d-1)^{2}>2d+1>2d+(\frac{e}{2}-1)s^{-d+1},$ it is suffices to prove that
$d+\frac{e}{2}s^{-d+1}>4\pi \sqrt{\frac{e}{2}+1} s^{\frac{-d+3}{2}}.$ Using $k\geq e+2\geq4$ and $d\geq7$, we obtain
$$s^{\frac{d-3}{2}}(d+\frac{e}{2}s^{-d+1})>(k-1)d\geq(e+1)d>4\pi \sqrt{\frac{e}{2}+1}.$$
\qed
\end{proof}


\end{document}